\documentclass[leqno,12pt]{amsart} 
\usepackage{amsmath, amssymb}
\usepackage{amsthm} 
\usepackage[all]{xy}
\renewcommand{\thefootnote}{} 
\theoremstyle{plain} 
\newtheorem{theorem}{\indent\sc Theorem}[section]
\newtheorem{lemma}[theorem]{\indent\sc Lemma}

\newtheorem{proposition}[theorem]{\indent\sc Proposition}

\theoremstyle{definition} 
\newtheorem{definition}[theorem]{\indent\sc Definition}
\newtheorem{remark}[theorem]{\indent\sc Remark}

\newtheorem{notation}[theorem]{\indent\sc Notation}

\newcommand{\R}{\mathbb{R}}

\newcommand{\N}{\mathbb{N}}
\newcommand{\Pen}{\mathrm{Pen}}

\newcommand{\supp}{\mathop{\mathrm{supp}}\nolimits}
\newcommand{\abs}[1]{\lvert#1\rvert}
\newcommand{\norm}[1]{\lVert#1\rVert}

\newcommand{\limone}{\varprojlim{\!}^1}
\newcommand{\U}{\mathcal{U}}

\newcommand{\Horo}{\mathcal{H}}

\newcommand{\EG}{\underline{E}G}
\newcommand{\EP}{\underline{E}P}
\newcommand{\Xaug}{X(G,\mathbb{P},\mathcal{S})}
\newcommand{\EX}{EX(G,\mathbb{P})}
\newcommand{\famP}{\mathbb{P}}

\newcommand{\vect}[1]{\mathbf{#1}}
\def\a{\mathfrak{a}}
\makeatletter
\def\address#1#2{\begingroup
\noindent\parbox[t]{7.8cm}{%
\small{\scshape\ignorespaces#1}\par\vskip1ex
\noindent\small{\itshape E-mail address}%
\/: #2\par\vskip4ex}\hfill%
\endgroup}%
\makeatother
\title{\uppercase{The coarse Baum-Connes conjecture for relatively
hyperbolic groups}}
\author{
\textsc{Tomohiro Fukaya, Shin-ichi Oguni} 
}
\date{} 
\begin{document}

\maketitle

\footnote{ 
T.Fukaya was supported by Grant-in-Aid for Young Scientists (B)
(23740049) from The Ministry of Education, Culture, Sports, Science and
Technology
}
\renewcommand{\thefootnote}{\fnsymbol{footnote}} 


\begin{abstract}
We study a group which is hyperbolic relative to a finite family of
 infinite subgroups. We show that the group satisfies the coarse
 Baum-Connes conjecture if each subgroup belonging to the family
 satisfies the coarse Baum-Connes conjecture and admits a
 finite universal space for proper actions.
 Especially, the group satisfies the analytic Novikov conjecture.
\end{abstract}

\section{Introduction} 
\label{conj}
 Let $X$ be a proper metric space.
 We say that $X$ satisfies the coarse Baum-Connes conjecture if the
 following coarse
 assembly map $\mu_X$ of $X$ is an isomorphism:
\[
 \mu_X \colon KX_*(X) \rightarrow K_*(C^*(X)).
\]
  If a countable group $G$ equipped with a proper invariant metric
 satisfies the coarse Baum-Connes conjecture, and if $G$ admits a finite
 $G$-simplicial complex which is a universal space for proper actions,
 then, by a descent principle, $G$ satisfies the analytic Novikov
 conjecture. For details, see~\cite[Theorem 8.4]{MR1399087} and also
 \cite[Theorem 12.6.3]{MR1817560}.

There are several studies on the coarse
Baum-Connes conjecture for relatively hyperbolic groups. Let $G$ be a
group which is hyperbolic relative to a finite family of infinite
subgroups $\famP =\{P_1,\dots, P_k\}$ . Osin~\cite{MR2181790}
showed that $G$ has finite asymptotic dimension if each subgroup $P_i$
has finite asymptotic dimension. Ozawa~\cite{MR2243738} showed that $G$
is exact if each subgroup $P_i$ is exact. Dadarlat and
Guentner~\cite{MR2364071} showed that $G$ is uniformly embeddable in a
Hilbert space if each subgroup $P_i$ is uniformly embeddable in a
Hilbert space. Due to Yu's works~\cite{MR1626745}\cite{MR1728880},
those results imply the coarse Baum-Connes conjecture for such groups.

In the present paper, we show the following:
\begin{theorem}
\label{main_theorem} Let $G$ be a finitely generated group and
 $\famP=\{P_1,\dots,P_k\}$ be a finite family of infinite subgroups.
 Suppose that $(G,\famP)$ is a relatively hyperbolic group. 
 If each subgroup $P_i$ satisfies
 the coarse Baum-Connes conjecture, and  admits a finite $P_i$-simplicial
 complex which is a universal space for proper actions, then $G$ 
 satisfies the coarse Baum-Connes conjecture.
\end{theorem}

We note that $G$ admits a finite $G$-simplicial complex which is a
universal space for proper actions (see
Appendix~\ref{appendix:coarse-baum-connes-for-hyp}).

Here we summarize the proof of Theorem~\ref{main_theorem}. Let $\Xaug$
be the augmented space obtained by attaching horoballs to the Cayley
graph $\Gamma(G,\mathcal{S})$ along the left cosets of subgroups $P\in
\famP$ where $\mathcal{S}$ is a finite generating set
(Definition~\ref{def:Horoball} and Definition~\ref{notation:1}). Since
$\Xaug$ is $\delta$-hyperbolic, $\Xaug$ satisfies the coarse Baum-Connes
conjecture. We fix an order on horoballs. Let $X_n$ be a subspace
obtained by removing the first $n-1$ horoballs from $\Xaug$
(Notation~\ref{notation:X_n}). By Mayer-Vietoris arguments, we show
inductively that $X_n$ satisfies the coarse Baum-Connes conjecture
(Section~\ref{sec:induction-part}). To study the coarse assembly map for
$X_\infty= \bigcap X_n$, which is coarsely equivalent to $G$, we need to
analyze the coarse $K$-homology of the projective limit. We might expect
a so-called Milnor exact sequence
\begin{align}
\label{eq:wrong}
 0\rightarrow \limone KX_{p+1}(X_n) \rightarrow KX_p(X_\infty) 
 \rightarrow \varprojlim KX_p(X_n) \rightarrow 0.
\end{align}
Unfortunately, ~(\ref{eq:wrong}) is not necessarily exact, in general. A
simple counterexample is given by $Y_n = \R \setminus [-n,n]$. Thus we
introduce a contractible space $\EX$. The
following isomorphism (Proposition~\ref{prop:K-homology-of-X}) is
crucial to the proof of Theorem \ref{main_theorem}:
\[
 KX_*(\Xaug) \cong K_*(\EX).
\]
Sections~\ref{sec:coarse-k-homology}
and~\ref{sec:contractible-model} are
devoted to a proof of 
this isomorphism.  For the projective limit of locally
compact Hausdorff spaces, there is a Milnor exact sequence in
$K$-homology
(Section~\ref{sec:projective-limit-part}).
Combining this with an exact sequence in $K$-theory of $C^*$-algebras
(Proposition~\ref{prop:Phillips}), we complete the proof.

\section{Coarse K-homology of the augmented space}
\label{sec:coarse-k-homology}
Let $G$ be a finitely generated group with a finite family of
infinite subgroups 
$\famP = \{P_1,\dots P_k\}$. 
Groves and Manning \cite{MR2448064} introduced a space obtained by
attaching ``combinatorial horoballs'' to $G$ along the left
cosets of subgroups $P\in \famP$. Their construction is suitable for
Mayer-Vietoris arguments to compute the coarse K-homology of $G$ in
terms of that of $P\in \famP$.
We review the construction and study the coarse K-homology of the
resulting space.
\subsection{The augmented space}
\label{sec:augmented-space}
\begin{definition}
\label{def:Horoball}
Let $(P,d)$ be a proper metric space.
{\itshape The combinatorial horoball} based on $P$, denoted by
 $\Horo(P)$, is the graph defined as follows:
\begin{enumerate}
 \item $\Horo(P)^{(0)} = P \times (\N\cup \{0\})$. 
 \item $\Horo(P)^{(1)}$ contains the following two type of edges:
       \begin{enumerate}
	\item For each $l\in \N\cup\{0\}$ and $p,q\in P$, 
	      if $0< d(p,q)\leq 2^{l}$
	      then there is a 
	      {\itshape horizontal edge} connecting $(p,l)$ and $(q,l)$.
	\item For each $l\in \N\cup\{0\}$ and $p\in P$, there is a 
	      {\itshape vertical edge} connecting $(p,l)$ and $(p,l+1)$.
       \end{enumerate}
\end{enumerate}
Here $\N$ denotes the set of positive integers. We endow $\Horo(P)$ with
 the graph metric. For a closed subset $I\subset \R$, let $\Horo(P;I)$ denote the
full subgraph of $\Horo(P)$ spanned by
 $P\times (I \cap(\N\cup\{0\}))$.
\end{definition}


Let $G$ be a finitely generated group with a finite family of
infinite subgroups 
$\famP = \{P_1,\dots P_k\}$. 
We take a finite generating set $\mathcal{S}$ for $G$. 
We assume that
$\mathcal{S}$ is symmetrized, so that $\mathcal{S} = \mathcal{S}^{-1}$.
We endow $G$ with the left-invariant word metric
$d_{\mathcal{S}}$ with respect to $\mathcal{S}$.
We choose a sequence
$g_{1},g_{2},\dots$ in $G$ such that for each $r \in
\{1,\dots,k\}$, the map 
$\N\rightarrow G/P_r: a \mapsto g_{ak+r}P_r$ is bijective.
For $i=ak+r \in \N$, let $P_{(i)}$ denote a subgroup $P_r$.
Thus the set of all cosets $\bigsqcup_{r=1}^k G/P_r$ is indexed by the map
$\N\ni i \mapsto g_iP_{(i)}$.
Each coset $g_iP_{(i)}$ has a proper metric $d_i$ which is the
restriction of $d_{\mathcal{S}}$. 
Let $\Gamma$ be the Cayley graph of $(G,\mathcal{S})$.
There exists a natural embedding 
$\psi_i\colon \Horo(g_iP_{(i)};\{0\}) \hookrightarrow \Gamma$ such that
$\psi_i(x,0) = x$ for all $x\in g_iP_{(i)}$. 
\begin{definition}
\label{notation:1}
{\itshape The augmented space } $\Xaug$ is obtained by pasting
 $\Horo(g_iP_{(i)})$
 to $\Gamma$ by $\psi_i$ for all $i\in \N$. 
Thus we can write it as follows:
\begin{align*}
 \Xaug &= \Gamma\cup \bigcup_{i\in \N} \Horo(g_iP_{(i)}).
\end{align*}
We endow $\Xaug$ with the graph metric.  For
positive integer $N$, set
 \begin{align*} 
 X(N) &= \Gamma\cup \bigcup_{i\in \N} \Horo(g_iP_{(i)};[0,N]);\\
 Y(N) &= \bigsqcup_{i\in \N} \Horo(g_iP_{(i)};[N,\infty));\\ 
 Z(N) &= \bigsqcup_{i\in \N} \Horo(g_iP_{(i)};\{N\}).
\end{align*}
\end{definition}
\begin{remark}
 The vertex set of $\Xaug$, denoted by $\Xaug^{(0)}$, can naturally be
 identified with the set of 2-tuple $(x,t)$, where $x\in
 \bigsqcup_{i} g_iP_{(i)}$ and $t\in \N$, or $x\in G$ and $t=0$.
 We endow $\Xaug^{(0)}$ with the metric from the graph structure.
\end{remark}
%
%
%
\begin{definition}
 The pair $(G,\famP)$ is a relatively hyperbolic group
 if the augmented space $\Xaug$ is $\delta$-hyperbolic for some 
 $\delta \geq 0$.
\end{definition}
\begin{remark}
 Groves and Manning~\cite[Theorem 3.25]{MR2448064} show that the above
 definition is equivalent to other various definitions.  See
 also~\cite{MR2684983}.
\end{remark}

\subsection{An anti-\v{C}ech system}
We form an anti-\v{C}ech system $\{\U(j)\}_j$ of $\Xaug^{(0)}$ 
as follows: For $i\geq 1, (x,t)\in g_iP_{(i)}\times \N$ and $j\geq 1$, 
a column centered at $(x,t)$ with the size $j$ is
\[
 B((x,t),j) = 
 \{(y,l) \in g_iP_{(i)}\times \N: d_{\mathcal{S}}(x,y) \leq 2^{t+j}, t\leq
 l\leq t+j\}.
\]
For $x\in G$ and $j\geq 1$, a column centered at $(x,0)$
with the size $j$ is 
\[
 B((x,0),j) = 
 \{(y,l) \in \Xaug^{(0)}: d_{\mathcal{S}}(x,y) \leq 2^{j}, 0\leq l\leq j\}.
\]
The locally finite cover $\U(j)$ is made up of all those columns with
size $j$, that is, 
\[
 \U(j) = \{B((x,t),j): (x,t) \in \Xaug^{(0)}\}.
\]
When $j\leq j'$, the map $\U(j)\rightarrow \U(j')$ is defined by sending
$B((x,t),j)$ to $B((x,t),j')$.

\subsection{Mayer-Vietoris sequences}
Set $j_n = 3^n, N_n = 3^n+1$ for $n\geq0 $.
We introduce a decomposition of $\U(j_n)$ as follows:
\begin{align*}
 \U_n &= \U(j_n);\\
 \mathcal{X}_n &= \{B\in \U(j_n) : B \cap X(N_n) \neq \emptyset\}; \\
 \mathcal{Y}_n &= \{B\in \U(j_n) : B \cap Y(N_n) \neq \emptyset\}; \\
 \mathcal{Z}_n &= \{B\in \U(j_n) : B \cap Z(N_n) \neq \emptyset\}; \\
 \mathcal{Z}_n^i &= \{B\in \mathcal{Z}_n : B\cap \Horo(g_iP_{(i)})
  \neq \emptyset \}. 
\end{align*}
We remark that $\U_n = \mathcal{X}_n \cup \mathcal{Y}_n,
\mathcal{X}_n\cap \mathcal{Y}_n = \mathcal{Z}_n$ and $\mathcal{Z}_n =
\bigsqcup_i \mathcal{Z}_n^i$.  Then the pair $(\mathcal{X}_n,
\mathcal{Y}_n)$ forms an excision pair of
$\U_n$ and the map $\U_n \rightarrow \U_{n+1}$ preserves the pairs. Thus
we have the following exact sequence:
\begin{align}
\label{eq:varinjelimMV}
\cdots \rightarrow \varinjlim K_p(\abs{\mathcal{Z}_n}) 
 \rightarrow \varinjlim K_p(\abs{\mathcal{X}_n}) 
 \oplus \varinjlim K_p(\abs{\mathcal{Y}_n}) 
 \rightarrow \varinjlim K_p(\abs{\U_n}) \rightarrow 
 \varinjlim K_{p-1}(\abs{\mathcal{Z}_n}) \rightarrow \cdots. 
\end{align}
Since $\{\U_n\}_n$ forms an anti-\v{C}ech system of $\Xaug^{(0)}$, 
we have $\varinjlim K_*(\abs{\U_n}) = KX_*(\Xaug)$. 
In this section, we compute $\varinjlim
K_*(\abs{\mathcal{X}_n})$ and $\varinjlim K_*(\abs{\mathcal{Y}_n})$. 

\begin{lemma}
\label{lem:X1}
 The inductive limit of $K_*(\abs{\mathcal{X}_n})$ is
 isomorphic to $KX_*(X(1))$.
\end{lemma}

\begin{proof}
For $N\geq j+1\geq 0$, we define that the subset $\U(N,j)$ of $\U(j)$ is
made up of all columns $B((x,t),j) \in \U(j)$ which intersect with $X(N)$.
We remark that $\mathcal{X}_n = \U(N_n,j_n)$.
We define simplicial maps $\alpha_n,\, \beta_n,\, \gamma_n$ by
\begin{align*}
 \alpha_n &\colon \U(1,j_n) \rightarrow \U(N_{n},j_{n}) &\colon& B((x,t),j_n)
 \mapsto B((x,t),j_{n}),\\
 \beta_n &\colon \U(N_n,j_n) \rightarrow \U(1,j_{n+1}) &\colon& B((x,t),j_n)
 \mapsto \begin{cases}
	  B((x,1),j_{n+1}) &(t\geq 1)\\
	  B((x,0),j_{n+1}) &(t=0),
	 \end{cases}\\
 \gamma_n &\colon \U(N_n,j_n)  \rightarrow \U(N_{n+1},j_{n+1}) 
 &\colon& B((x,t),j_n) \mapsto B((x,t),j_{n+1}).
\end{align*}
Clearly $\alpha_{n+1} \circ \beta_n$ and $\gamma_n$ belong to the same
contiguity class.
Since two simplicial maps belonging to the same contiguity class define
continuous maps which are homotopic \cite[Lemma 5.5.2.]{MR666554}, we
have the following commutative diagram:
\begin{align*}
 \xymatrix{
 K_*(\abs{\U(1,j_n)}) \ar[r]^{{\alpha_n}_*} \ar[d] & 
   K_*(\abs{\U(N_n,j_n)}) \ar[dl]_{{\beta_n}_*} \ar[d]_{{\gamma_n}_*}\\
 K_*(\abs{\U(1,j_{n+1})}) \ar[r]^{{\alpha_{n+1}}_*} & 
   K_*(\abs{\U(N_{n+1},j_{n+1})}) .
}
\end{align*}
It follows that 
$\varinjlim K_*(\abs{\U(1,j_n)}) \cong \varinjlim K_*(\abs{\U(N_n,j_n)})$.

Let $\U(1,j_n)\cap X(1)$ denote the cover of $X(1)$ which consists of all
$B\cap X(1)$ for $B\in \U(1,j_n)$. Then $\{\U(1,j_n)\cap X(1)\}_n$ forms
an anti-\v{C}ech system of $X(1)$. 
Since $\abs{\U(1,j_n)\cap X(1)}$ = $\abs{\U(1,j_n)}$, we have 
$KX_*(X(1)) = \varinjlim K_*(\abs{\mathcal{X}_n})$.
\end{proof}

\begin{lemma}
 The inductive limit of $K_*(\abs{\mathcal{Y}_n})$ is trivial.
\end{lemma}
\begin{proof}
For an integer $s\geq 0$, 
we define a simplicial map $q_{n,s}\colon \mathcal{Y}_n
 \rightarrow \mathcal{Y}_{n+1}$ by 
\begin{align*}
 q_{n,s}(B((x,t),j_n)) = \begin{cases}
		  B((x,t),j_{n+1}) & \text{ if } t\geq s,\\
		  B((x,s),j_{n+1}) & \text{ if } t < s.
		 \end{cases}
\end{align*}
Clearly $q_{n,s}$ and $q_{n,s+1}$ are contiguous. Let $h_{n,s}\colon
[s,s+1]\times \abs{\mathcal{Y}_n} \rightarrow \abs{\mathcal{Y}_{n+1}}$
be a proper homotopy between geometric realizations of $q_{n,s}$ and
$g_{n,s+1}$.  We define a proper map $q_n \colon \R_{\geq 0} \times
\abs{\mathcal{Y}_n} \rightarrow \abs{\mathcal{Y}_{n+1}}$ by 
$q_n(\theta,x) = h_{n,\lfloor \theta \rfloor}(\theta,x)$, 
where 
$\theta\in \R_{\geq 0}$, $x\in \abs{\mathcal{Y}_n}$, 
and $\lfloor \theta \rfloor$
denotes the largest integer not greater than $\theta$.  Then we have
the following commutative diagram:
\begin{align*}
 \xymatrix{
 \abs{\mathcal{Y}_n} \ar[rr] \ar@{^{(}->}[rd] && \abs{\mathcal{Y}_{n+1}}\\
 & \R_{\geq0} \times \abs{\mathcal{Y}_n} \ar^{q_n}[ru]
 }
\end{align*}
Here the horizontal arrow is the canonical map and the map 
$\abs{\mathcal{Y}_n} \hookrightarrow \R_{\geq0} \times
 \abs{\mathcal{Y}_n}$ 
is given by the inclusion onto $\{0\}\times \abs{\mathcal{Y}_n}$. 
 Since $\R_{\geq0} \times \abs{\mathcal{Y}_n}$ is contractible
 (see \cite[Remark 7.1.4]{MR1817560}), the  homomorphism 
 $K_*(\abs{\mathcal{Y}_n}) \rightarrow K_*(\abs{\mathcal{Y}_{n+1}})$ 
 factors through zero. Therefore, 
 $\varinjlim K_*(\abs{\mathcal{Y}_n}) = 0$.
\end{proof}

By the cluster axiom of K-homology 
(see \cite[Definition 7.3.1]{MR1817560}), we
have $K_*(\abs{\mathcal{Z}_n}) \cong \prod_{i\geq 1}
K_*(\abs{\mathcal{Z}_n^i})$. 
Therefore we have the following exact sequence: 
\begin{align}
\label{eq:injlim-Mayer-Vietoris}
\cdots \rightarrow \varinjlim \prod_{i\geq 1} K_*(\abs{\mathcal{Z}_n^i}) 
 \rightarrow KX_*(X(1)) 
 \rightarrow KX_*(\Xaug) \rightarrow \cdots.
\end{align}
We remark that $KX_*(X(1)) \cong KX_*(G)$ since $X(1)$ and $G$ are
coarsely equivalent. In the next section, we will show 
$\varinjlim \prod_{i\geq 1} K_*(\abs{\mathcal{Z}_n^i}) \cong
\prod_{i\geq 1}KX_*(g_iP_{(i)})$ with the aid of finite universal spaces 
$\EP_1,\dots,\EP_k$.

\section{Contractible models}
\label{sec:contractible-model}

In this section, we take $(G,\famP)$ in Theorem~\ref{main_theorem}.  Let
$\EG$ be a finite $G$-simplicial complex which is a universal space for
proper actions. For $r\in \{1,\dots,k\}$, let $\EP_r$ be a finite
$P_r$-simplicial complex which is a universal space for proper
actions. In the rest of this paper, we assume that all $\EP_{r}$ are
embedded in $\EG$.  We also assume that $G$ is naturally embedded in the
set of vertices of $\EG$ and $g_iP_{(i)}$ is embedded in $g_i\EP_{(i)}$.
If $(G,\famP)$ satisfies conditions in Theorem~\ref{main_theorem}, then
we can take $\EG$ satisfying these conditions (see
Appendix~\ref{appendix:finite-univ-space}). We take a finite subcomplex
$\Delta\subset \EG$ containing a fundamental domain of $\EG$.  We may assume
that $\Delta\cap \EP_r$ contains a fundamental domain of $\EP_r$ for
$r=1,\dots,k$ without loss of generality.


Now, we introduce a contractible model of $\Xaug$.
We define an embedding 
$\varphi_i \colon g_i\EP_{(i)}\times \{0\} \hookrightarrow \EG$ 
by $\varphi_i(x,0) = x$.

A contractible model for $\Xaug$ is obtained by pasting
$g_i\EP_{(i)}\times [0,\infty)$ to $\EG$ by $\varphi_i$ for all 
$i\in \N$. Thus we can write it as follows:
\begin{align*}
 \EX &= \EG\cup \bigcup_{i\in \N} (g_i\EP_{(i)}\times [0,\infty)).
\end{align*}
Contractible models for $X(1), Y(1)$ and
$\Horo(g_iP_{(i)};\{1\})$ are also defined as follows:
\begin{align*}
 EX(1) &= \EG\cup \bigcup_{i\in \N}  (g_i\EP_{(i)}\times [0,1]);\\
 EY(1) &= \bigsqcup_{i\in \N} (g_i\EP_{(i)}\times [1,\infty));\\
 EZ^i &= g_i\EP_{(i)}\times \{1\}.
\end{align*} 

%

We remark that $\EX$ admits a proper metric such that $\EX$ is coarsely
equivalent to $\Xaug$, but it is neither of bounded geometry nor 
uniformly contractible, if $\famP$ is not empty. 
Thus $\EX$ is not coarsening of $\Xaug$ in the sense of 
\cite[Definition 2.4]{MR1399087}. However $\EX$ is a ``weakly coarsening''
of $\Xaug$ in the following sense:
\begin{proposition}
\label{prop:K-homology-of-X}
 The coarse K-homology of $\Xaug$ can be computed by 
 the contractible model, that is, 
 $KX_*(\Xaug) \cong K_*(\EX)$.
\end{proposition}
Proposition
 \ref{prop:K-homology-of-X} is no direct
 consequence of ~\cite[Proposition 3.8]{MR1388312}. Our strategy is
 cutting off horoballs by Mayer-Vietoris arguments.

\subsection{Proof of Proposition~\ref{prop:K-homology-of-X}}
\label{sec:proof-prop-refpr}
We construct a locally
finite cover $E\U_n$ of $\EX$ as follows:
for $x\in g_iP_{(i)}$ and $j\geq 1$, 
the ball in $g_i\EP_{(i)}$ centered at
$x$ with
the size $j$ is
\begin{align}
\label{def:EB}
  EB(x,j) = \bigcup y(\Delta\cap \EP_{(i)})
\end{align}
where the union is taken over all $y \in g_iP_{(i)}$ such that
$d_{\mathcal{S}}(x,y) \leq 2^{j}$.
A contractible column
centered at $(x,t) \in g_iP_{(i)} \times \N$ with the size $j$ is 
\[
 EB((x,t),j) = EB(x,t+j) \times [t,t+j].
\]
For $x\in G$, a contractible column centered at 
$(x,0) \in G\times \{0\}$ with the size $j$ is 
\[
 EB((x,0),j) = \bigcup 
\left(y\Delta \cup \bigcup_{i \in \N}\big((y\Delta\cap g_i\EP_{(i)} ) 
\times [0,j]\big)\right)
\]
where the first union is taken over all $y\in G$ such that
$d_{\mathcal{S}}(x,y)\leq 2^j$.  We define that the cover $E\U_n$ of
$\EX$ consists of all those columns $EB((x,t),j_n)$ for 
$(x,t)\in \Xaug^{(0)}$.  
Taking subsequence if necessary,
we define a simplicial map $E\U_n \rightarrow \U_{n+1}$ by
$EB((x,t),j_n) \mapsto B((x,t),j_{n+1})$.

A partition of the unity gives a continuous map
$h_n\colon \EX \rightarrow \abs{E\U_n}$. The composite of $h_2$ and 
$\abs{E\U_2}\rightarrow \abs{\U_3}$
induces a homomorphism $K_*(\EX)\rightarrow KX_*(\Xaug)$.

Next, 
for each $i \in \N$, we construct an anti-\v{C}ech system
$\{E\mathcal{Z}_n^i\}_n$  
of $EZ^i$ as follows:
the cover $E\mathcal{Z}_n^i$ of $EZ^i$
consists of all balls $EB(x,j_n)\times \{1\}$ for $x\in g_iP_{(i)}$.  
Then $\{E\mathcal{Z}_n^i\}_n$ forms an anti-\v{C}ech system.

We define a simplicial map 
$\mathcal{Z}_n^i \rightarrow E\mathcal{Z}_{n+1}^i$ by 
$B((x,s),j_n) \mapsto EB(x,j_{n+1})\times \{1\}$. 
We also define a simplicial map 
$E\mathcal{Z}_n^i \rightarrow \mathcal{Z}_{n+1}^i$ by 
$EB(x,j_n)\times \{1\} \mapsto B((x,1),j_{n+1})$. Then we have a
commutative diagram
\begin{align*}
 \xymatrix{
 \prod_{i \in \N} K_*(\abs{\mathcal{Z}_n^i}) \ar[r] \ar[d]&
 \prod_{i \in \N} K_*(\abs{E\mathcal{Z}_{n+1}^i}) \ar[d] \ar[ld]\\
 \prod_{i \in \N} K_*(\abs{\mathcal{Z}_{n+2}^i}) \ar[r] &
 \prod_{i \in \N} K_*(\abs{E\mathcal{Z}_{n+3}^i}).
 }
\end{align*}
Hence $\varinjlim \prod_{i \in \N} K_*(\abs{\mathcal{Z}_n^i})
\cong \varinjlim \prod_{i \in \N} K_*(\abs{E\mathcal{Z}_n^i})$.
The partition of the unity gives a continuous map
$h^i_{n}\colon EZ^i \rightarrow \abs{E\mathcal{Z}_n^i}$ for
$i$ and $n\geq 1$.
By the
proof of \cite[Proposition 3.8]{MR1388312}, 
taking a subsequence if
necessary (not depending on $i$), 
the induced map 
$(h^i_{n})_*\colon K_*(EZ^i) \rightarrow K_*(\abs{E\mathcal{Z}_n^i})$ 
is an isomorphism onto the image of the map 
$K_*(\abs{E\mathcal{Z}_{n-1}^i}) \rightarrow K_*(\abs{E\mathcal{Z}_{n}^i})$.
See also \cite[Lemma 7.11]{decomposition-complexity}. 
It follows that 
\begin{align}
\label{eq:products_commutes_limit}
  \prod_{i \in \N}K_*(EZ^i) \cong
  \varinjlim \prod_{i \in \N} K_*(\abs{E\mathcal{Z}_n^i}) 
  \cong \varinjlim \prod_{i \in \N} K_*(\abs{\mathcal{Z}_n^i}). 
\end{align}

By arguments similar to that in the case of $EZ^i$, we can 
show the following isomorphism:
\begin{align}
\label{eq:EG}
 K_*(EX(1)) \cong \varinjlim K_*(\abs{\mathcal{U}(1,j_n)}) = KX_*(X(1)).
\end{align}


By the Mayer-Vietoris sequence for $\EX = EX(1) \cup EY(1)$, the exact
sequence~(\ref{eq:injlim-Mayer-Vietoris})
and the fact that $K_*(EY(1)) = 0$,
we have the following commutative diagram with two horizontal
exact sequences:
\begin{align}
\xymatrix{
\ar[r] & \prod_{i \in \N} K_*(EZ^i(1)) \ar[r] \ar[d]
 & K_*(EX(1)) \ar[r] \ar[d]
 & K_*(\EX) \ar[r] \ar[d] & \\
\ar[r] &  \varinjlim \prod_{i \in \N} K_*(\abs{\mathcal{Z}_n^i})\ar[r]
 & KX_*(X(1)) \ar[r]
 & KX_*(\Xaug) \ar[r] &.
}
\end{align}
By ~(\ref{eq:products_commutes_limit}),~(\ref{eq:EG}) and the five lemma,
all vertical maps are isomorphisms. This completes the proof of
Proposition~\ref{prop:K-homology-of-X}.

\section{Coarse Mayer-Vietoris sequences}
Higson, Roe and Yu~\cite{MR1219916} introduced a coarse Mayer-Vietoris
sequence in the K-theory of the Roe algebras. It is used to prove a
Lipschitz homotopy invariance of the K-theory of the Roe
algebras~\cite[Theorem 9.8]{MR1399087}. 


We first recall a notion of ``excision pair'' in coarse category. 
For a metric space $M$, a subspace $A$, and a positive number $R$, we
denote by $\Pen(A;R)$ the $R$-neighbourhood of $A$ in $M$, that is,
$\Pen(A;R) = \{p\in M: d(p,A)\leq R\}$.
\begin{definition}
 Let $M$ be a proper metric space, and let $A$ and $B$ be closed
 subspaces with $M = A\cup B$. We say that 
 $M= A\cup B$ is an $\omega$-excisive
 decomposition, if for each $R> 0$ there exists some $S> 0$ such that 
\[
 \Pen(A;R)\cap \Pen(B;R) \subset \Pen(A\cap B; S).
\]
\end{definition}
We summarize results in \cite{MR1219916} (see also \cite{MR1834777} and
\cite{MR2012966}) on
 coarse assembly maps and Mayer-Vietoris sequences as follows:
\begin{theorem}
 Suppose that $M=A \cup B$ is an $\omega$-excisive decomposition. 
Then the following diagram is commutative and horizontal sequences are exact:
\begin{align*}
\xymatrix@1{
\ar[r]& KX_p(A\cap B) \ar[r] \ar[d]
 & KX_p(A) \oplus KX_p(B) \ar[r] \ar[d]
 & KX_p(M) \ar[r] \ar[d] & KX_{p-1}(A\cap B) \ar[r] \ar[d] &\\
\ar[r] & K_p(C^*(A\cap B)) \ar[r]
 & K_p(C^*(A))\oplus K_p(C^*(B)) \ar[r]
 & K_p(C^*(M)) \ar[r] & K_{p-1}(C^*(A\cap B)) \ar[r]. &
}
\end{align*}
Here vertical arrows are coarse assembly maps.
\end{theorem}

\section{Proof of theorem~\ref{main_theorem}}


In this section, we give a proof of Theorem~\ref{main_theorem}, which is
divided into two parts. In the first part, we show inductively the coarse
Baum-Connes conjecture for the space obtained by removing the first $n-1$
horoballs from $\Xaug$. In the second part, we compute the coarse
K-homology and the K-theory of the Roe algebra
of $G$ which is the intersection of a decreasing sequence of subspaces
of $\Xaug$.
\subsection{The first part}
\label{sec:induction-part}
\begin{notation}
\label{notation:X_n} 
 We introduce the following notations:
\begin{align*}
 X_n &= \Gamma \cup \bigcup_{i \geq n}\Horo(g_iP_{(i)});\\
 X_\infty &= \bigcap_{n\geq 1} X_n;\\
 EX_n &= 
   \EG \cup \bigcup_{i \geq n}
     (g_i\EP_{(i)}\times [0,\infty))\\
 EX_\infty &= \bigcap_{n\geq 1} EX_n.
\end{align*}
\end{notation}

We remark that $X_1 = \Xaug$, $X_\infty = \Gamma$, 
$EX_1 =\EX$ and $EX_\infty = \EG$.

Since $X_1$ is $\delta$-hyperbolic for some $\delta \geq 0$, by the result of
Higson-Roe~\cite[Corollary 8.2]{MR1388312}, the coarse assembly map
$\mu \colon KX_*(X_1) \rightarrow K_*(C^*(X_1))$ is an isomorphism. 
See Appendix~\ref{appendix:coarse-baum-connes-for-hyp}.
In fact, by Proposition~\ref{prop:K-homology-of-X}, the coarse assembly map
\begin{align}
\label{eq:3}
 \mu \colon K_*(EX_1) \rightarrow K_*(C^*(X_1))
\end{align}
is an isomorphism.
By assumption and \cite[Proposition 3.8]{MR1388312}, $\mu\colon
K_*(g_n\EP_{(n)})\rightarrow K_*(C^*(g_nP_{(n)}))$ is an isomorphism for
all $n \geq 1$.
\begin{lemma}
\label{lem:induction}
For any $n\geq 0$, the coarse assembly map 
$\mu_n\colon K_*(EX_n)\rightarrow K_*(C^*(X_n))$ is an isomorphism.
\end{lemma}
\begin{proof}
 We assume that $\mu_n$ is an isomorphism. 
 Since 
$X_n = X_{n+1}\cup \Horo(g_{n}P_{(n)})$
 is an $\omega$-excisive 
 decomposition, it follows from
 (coarse) Mayer-Vietoris sequences and the five lemma 
 that $\mu_{n+1}$ is an isomorphism.
\end{proof}

\subsection{The second part}
\label{sec:projective-limit-part}
Let $(EX_n)^{+}$ denote the one-point compactification of $EX_n$.
 It is clear that $(EX_\infty)^{+} = \bigcap_{n\in \N} (EX_n)^{+}$.
By the Milnor exact sequence \cite[Proposition7.3.4]{MR1817560}, we have
\begin{align}
\label{eq:Milnor-K-homology}
  0\rightarrow \limone K_{p+1}((EX_n)^{+}) \rightarrow K_p((EX_\infty)^{+}) 
 \rightarrow \varprojlim K_p((EX_n)^{+}) \rightarrow 0.
\end{align}
Since the K-homology of $EX_n$ is just the reduced K-homology of
$(EX_n)^{+}$, we have $K_*((EX_n)^{+}) \cong K_*(EX_n) \oplus
K_*(\{+\})$ where $\{+\}$ denotes a one-point space. This is also a
direct consequence of an exact 
sequence~\cite[Definition7.1.1(b)]{MR1817560}. Thus we can replace
$K_*((EX_n)^{+})$ in (\ref{eq:Milnor-K-homology}) by $K_*(EX_n)$.


Next, we consider the K-theory of the Roe algebras. 
Let $\mathrm{H}$ be a Hilbert space and $\rho\colon C_0(X_1)
\rightarrow \mathfrak{B}(\mathrm{H})$ is an ample representation where
$\mathfrak{B}(\mathrm{H})$ is the set of all bounded operators on
$\mathrm{H}$.
The Roe algebra
$C^*(X_1,\mathrm{H})$ is the norm closure of the algebra of locally
compact, controlled operators on $\mathrm{H}$ 
(see \cite[Definition 6.3.8]{MR1817560}). The restriction
$\rho\colon C_0(X_n) \rightarrow
\mathfrak{B}(\overline{C_0(X_n)\mathrm{H}})$ gives 
an ample representation of $C_0(X_n)$. The Roe algebra
$C^*(X_n,\overline{C_0(X_n)\mathrm{H}})$ can be naturally identified with
a sub-$C^*$-algebra of $C^*(X_1,\mathrm{H})$, in fact, we have
\[
 C^*(X_n,\overline{C_0(X_n)\mathrm{H}}) = \{T\in C^*(X_1,\mathrm{H}): 
 \supp T \subset X_n\times X_n\}.
\]
We abbreviate $C^*(X_n,\overline{C_0(X_n)\mathrm{H}})$ to $C^*(X_n)$.
Now it is easy to see that $C^*(X_{\infty}) = \bigcap_{n\geq 1}C^*(X_{n}).$

Phillips~\cite{MR1050490} studied the
K-theory of the projective limit of $C^*$-algebras.
\begin{proposition}[{\cite[Theorem 5.8(5)]{MR1050490}}]
\label{prop:Phillips}
The following sequence is exact. 
\begin{align*}
 \label{eq:2}
0\rightarrow \limone K_{p+1}(C^*(X_n)) \rightarrow
 K_p(C^*(X_\infty))
 \rightarrow  \varprojlim K_p(C^*(X_n)) \rightarrow  0.
\end{align*}
\end{proposition}
By Proposition~\ref{prop:Phillips} and (\ref{eq:Milnor-K-homology}),
 we have the following commutative diagram such that upper and lower
 horizontal sequences are exact:
\begin{eqnarray*}
 \xymatrix{
   0  \ar[r]& \limone K_{p+1}(EX_n) \ar[d] \ar[r]
    & K_p(EX_\infty) \ar[d]\ar[r] 
    & \varprojlim K_p(EX_n) \ar[d] \ar[r]
    & 0 .\\
  0 \ar[r]& \limone K_{p+1}(C^*(X_n)) \ar[r]& 
   K_p(C^*(X_\infty)) \ar[r]
    & \varprojlim K_p(C^*(X_n)) \ar[r] & 0.
}
\end{eqnarray*}

By Lemma~\ref{lem:induction} and the five lemma, every vertical map is
an isomorphism. This completes the proof of Theorem~\ref{main_theorem}.

\begin{remark}
 In the proof of Theorem~\ref{main_theorem}, we use
 $\delta$-hyperbolicity of the augmented space only for the first step
 of the induction in section~\ref{sec:induction-part} and the existence
 of a universal space $\EG$ mentioned in the beginning of
 Section~\ref{sec:contractible-model}.
\end{remark}

\appendix
\section{A finite universal space for proper actions of a relatively
 hyperbolic group}
\label{appendix:finite-univ-space}
In this appendix we prove the following 
(refer to~\cite[Theorem 0.1]{MR1974394} on the case of torsion free
groups):
\begin{theorem}\label{finite}
Let a countable group 
$G$ be hyperbolic relative to a finite family of infinite subgroups 
$\famP$. 
Suppose that every $P\in\famP$ admits a finite $P$-simplicial complex 
which is a universal space for 
proper actions.
Then $G$ admits a finite $G$-simplicial complex 
which is a universal space for 
proper actions. In fact, $G$ has a  finite
$G$-simplicial complex $\EG$ with an embedding 
$i\colon G\hookrightarrow \EG$ and each $P\in \famP$ has a finite
 $P$-simplicial complex $\EP$ which is a subcomplex of $\EG$ such that 
$i(P) \subset \EP$.
\end{theorem}
\noindent
See~\cite{MR2195456} for universal spaces for proper actions. 

Let a countable group $G$ be finitely generated relative to
a finite family of infinite subgroups $\famP$. 
We denote the family of all left cosets by $\a:=\bigsqcup_{P\in\famP}G/P$.
We take a left invariant, proper metric $d_G$ on $G$ such that 
$G$ is generated by 
$\{g\in G \ |\ d_G(e,g)\le 1\}\cup\bigcup_{P\in\famP}P$.
We remark that  $\{g\in G \ |\ d_G(e,g)\le 1\}$
is a finite set.

Now we recall the definition of
the augmented space 
$X(G,\famP,d_G)$ (see~\cite[Section 3]{MR2448064} and
also~\cite{MR2684983}).
Its vertex set $V(G,\famP,d_G)$ is 
$G\sqcup \bigsqcup_{A\in\a}(A\times \N)$ where $\N$ is the set of
positive integers.
We often denote the subset $G\subset
V(G,\famP)$ by $G\times \{0\}$.
Also we often regard $A\in \a$ as a subset $A\times \{0\}$ of 
$G\times \{0\}$.  
Its edge is either a vertical edge or a horizontal edge: 
a vertical edge is a pair 
$\{(a,t_1),(a,t_2)\}\subset A\times (\{0\}\sqcup\N)$
such that $|t_1-t_2|=1$ for $A\in \a$; 
a horizontal edge is a pair 
$\{(a_1,t),(a_2,t)\}\subset A\times \N$
such that $0<d_G(a_1,a_2)\le 2^t$ for $A\in \a$
or a pair of $\{g_1,g_2\}\subset G$
such that $d_G(g_1,g_2) = 1$.  

Since $G$ is generated by 
$\{g\in G \ |\ d_G(e,g)\le 1\}\cup\bigcup_{P\in\famP}P$,
the augmented space $X(G,\famP,d_G)$ is connected.
This graph structure induces a metric on $V(G,\famP,d_G)$.
When we consider for $P\in\famP$, 
a left invariant proper metric $d_P:=d_G|_{P\times P}$ on $P$, 
then
$X(P,\{P\},d_P)$ is nothing but the full
subgraph of $P\sqcup (P\times \N)$ in $X(G,\famP,d_G)$.  
Moreover we can confirm that 
$X(P,\{P\},d_P)$ is an 
isometrically embedded subgraph of 
$X(G,\famP,d_G)$.

We consider the Rips complex $R_D(V(G,\famP,d_G))$ for a positive integer $D$. 
We denote 
the full subcomplexes of 
\begin{align*}
V(G,\famP,d_G)_r&=\bigsqcup_{A\in\a}(A\times \{r,\ldots\}); \\
V(G,\famP,d_G)^R&=G\sqcup \bigsqcup_{A\in\a}(A\times \{1,\ldots,R\}); \\
V(G,\famP,d_G)_r^R&=\bigsqcup_{A\in\a}(A\times \{r,\ldots,R\})
 =V(G,\famP,d_G)_r\cap V(G,\famP,d_G)^R,
\end{align*}
in $R_D(V(G,\famP,d_G))$ by 
$R_D(V(G,\famP,d_G))_r$, $R_D(V(G,\famP,d_G))^R$ and $R_D(V(G,\famP,d_G))_r^R$, 
respectively, 
where $r, R\in \N$ such that $r\le R$.

\begin{remark}\label{D}
If $r+D\le R$, then we have 
$R_D(V(G,\famP,d_G))=R_D(V(G,\famP,d_G))_r\cup R_D(V(G,\famP,d_G))^R$
and $R_D(V(G,\famP,d_G))_r^R=R_D(V(G,\famP,d_G))_r\cap R_D(V(G,\famP,d_G))^R$. 
\end{remark}

$G$ is hyperbolic relative to $\famP$
if and only if $V(G,\famP,d_G)$ is $\delta$-hyperbolic for some
$\delta\ge 0$ (see~ \cite[Theorem 3.25]{MR2448064}). 
Since $V(G,\famP,d_G)$ is $\delta$-hyperbolic, 
there exists some positive number $D_\delta$
such that for any $D\in \N$ such that $D\ge D_\delta$, 
the Rips complex $R_D(V(G,\famP,d_G))$ 
is contractible. 
Moreover we have the following: 
\begin{proposition}\label{Rips}
Let a countable group 
$G$ be hyperbolic relative to a finite family of infinite subgroups 
$\famP$. 
Suppose that $V(G,\famP,d_G)$ is $\delta$-hyperbolic, 
where $\delta$ is a non-negative number.
Then there exists some positive number $D'_\delta$
such that for any integer $D$ such that $D\ge D'_\delta$, the first
 barycentric subdivision of 
the Rips complex $R_D(V(G,\famP,d_G))$ 
is a $G$-simplicial complex 
which is a universal space for 
proper actions.
\end{proposition}
\noindent
If $\famP$ is empty on the above, 
then $G$ is a hyperbolic group. 
The above for this case is known (\cite{MR1887695}). 
Since arguments in the proof of~\cite[Theorem 1]{MR1887695}
can be applied to the above, we omit its proof.

\begin{proof}[Proof of Theorem \ref{finite}]
We take a left invariant proper metric $d_G$ on $G$ such that 
$G$ is generated by $\{g\in G \ |\ d_G(e,g)\le 1\}\cup\bigcup_{P\in\famP}P$. 
We denote by $d_P$ a left invariant proper metric $d_G|_{P\times P}$ on
 $P\in\famP$.

Suppose that $V(G,\famP,d_G)$ is $\delta$-hyperbolic. 
Then for every $P\in\famP$, the vertex set
$V(P,\{P\},d_P)$ is $\delta$-hyperbolic
because $X(P,\{P\},d_P)$ is an 
isometrically embedded subgraph of $X(G,\famP,d_G)$. 
We fix $D\in\N$ such that $D\ge D'_\delta$, 
where $D'_\delta$ is a constant in Proposition \ref{Rips}.
We take $P\in \famP$ and $r,R\in \N$ such that $r+D \le R$.
Also we take for every $P\in\famP$, 
a finite $P$-simplicial complex $\underline{E}P$ which is 
a universal space for proper actions.
Since the first barycentric subdivision of $R_D(V(P,\{P\},d_P))_r$ is a $P$-simplicial complex which is 
a universal space for proper actions by Proposition \ref{Rips}, 
we have a $P$-homotopy equivalent map 
$h_P:R_D(V(P,\{P\},d_P))_r\to \underline{E}P$.
It follows from an equivariant version of simplicial approximation theorem 
(see \cite[Exercise 6 for Chapter 1]{MR0413144})
that there exist a natural number $n$ 
and a $P$-simplicial map $f_P:R^{(n)}_D(V(P,\{P\},d_P))_r\to \underline{E}P$ 
which is $P$-homotopy equivalent to $h_P$ 
where $R^{(n)}_D(V(P,\{P\},d_P))_r$ is the $n$-th barycentric subdivision
of $R_D(V(P,\{P\},d_P))_r$.
We can take $n$ independently of $P$ because
 $\famP$ is a finite family.
We consider mapping cylinders 
\begin{align*}
(R^{(n)}_D(V(P,\{P\},d_P))_r^R\times [0,1])&
  \cup_{j_P} R^{(n)}_D(V(P,\{P\},d_P))_r;\\
(R^{(n)}_D(V(P,\{P\},d_P))_r^R\times [0,1])&\cup_{q_P} \underline{E}P,
\end{align*}
whose pasting maps are 
\begin{align*}
j_P&:R^{(n)}_D(V(P,\{P\},d_P))_r^R\times \{1\}\ni (x,1)\mapsto 
   x\in R^{(n)}_D(V(P,\{P\},d_P))_r;\\
q_P&:R^{(n)}_D(V(P,\{P\},d_P))_r^R\times \{1\}\ni (x,1)\mapsto 
   f_P(x)\in \underline{E}P,
\end{align*}
respectively. 
Then the maps $id_{R^{(n)}_D(V(P,\{P\},d_P))_r^R}$ and
 $f_P$ induce a map
\begin{align*}
 \widetilde{f_P}:& (R^{(n)}_D(V(P,\{P\},d_P))_r^R\times [0,1])
 \cup_{j_P}R^{(n)}_D(V(P,\{P\},d_P))_r\rightarrow \\
 & (R^{(n)}_D(V(P,\{P\},d_P))_r^R\times [0,1])\cup_{q_P}\underline{E}P,
\end{align*}
which is a $P$-homotopy equivalent map.
In fact we can confirm that 
$\widetilde{f_P}$ is 
a $P$-homotopy equivalent map relative to 
$R^{(n)}_D(V(P,\{P\},d_P))_r^R\times \{0\}$.
Now we construct two $G$-simplicial complex $R^{(n)}_D(V(G,\famP,d_G))_1$ and
 $R^{(n)}_D(V(G,\famP,d_G))_2$ as follows: 
First, $R^{(n)}_D(V(G,\famP,d_G))_1$ is obtained
 by, for every $P\in \famP$, pasting $G$-equivariantly, 
$(R^{(n)}_D(V(P,\{P\},d_P))_r^R\times [0,1])\cup_{j_P}
 R^{(n)}_D(V(P,\{P\},d_P))_r,$ to $R^{(n)}_D(V(G,\famP,d_G))^R$ 
by the pasting map
\[
 R^{(n)}_D(V(P,\{P\},d_P))_r^R\times \{0\} 
 \rightarrow R^{(n)}_D(V(P,\{P\},d_P))_r^R.
\]
Second, $R^{(n)}_D(V(G,\famP,d_G))_2$ is obtained by, for every $P\in \famP$,
 pasting $G$-equivariantly, 
$R^{(n)}_D(V(P,\{P\},d_P))_r^R\times [0,1]\cup_{q_P} \underline{E}P$ to
 $R^{(n)}_D(V(G,\famP,d_G))^R$ by the same pasting map.
Then they are $G$-homotopy equivalent by the induced map by
$id_{R^{(n)}_D(V(G,\famP,d_G))^R}$ and $\widetilde{f_P}$ for any $P\in\famP$.  Since
$R^{(n)}_D(V(G,\famP,d_G))$ is clearly $G$-homotopic to
 $R^{(n)}_D(V(G,\famP,d_G))_1$ by
Remark \ref{D}, we have
$R^{(n)}_D(V(G,\famP,d_G))$ is $G$-homotopic to
$R^{(n)}_D(V(G,\famP,d_G))_2$.  It follows from Proposition \ref{Rips} that
$R^{(n)}_D(V(G,\famP,d_G))_2$ is a $G$-simplicial complex which is a universal
space for proper actions. It is also 
clear that $R^{(n)}_D(V(G,\famP,d_G))_2$ is a finite $G$-simplicial complex by
the construction. $G$ is naturally embedded in
 $R^{(n)}_D(V(G,\famP,d_G))_2$. $R^{(n)}_D(V(P,\{P\},d_P))_2$ is a
 subcomplex of $R^{(n)}_D(V(G,\famP,d_G))_2$ and is a finite universal
 $P$-simplicial complex with the natural embedding of $P$.
\end{proof}

\section{The coarse Baum-Connes conjecture for hyperbolic metric spaces}
\label{appendix:coarse-baum-connes-for-hyp}
Higson and Roe \cite[Corollary 8.2]{MR1388312} proved the coarse Baum-Connes conjecture for hyperbolic
metric spaces. The following Proposition~\ref{prop:coarsenin_map} plays
an important role in their proof. 
\begin{proposition}
\label{prop:coarsenin_map}
 Let $Y$ be a compact metric space and let $\mathcal{O}Y$ denote an
 open cone of $Y$. Then the coarsening map
\[
 \mu \colon K_*(\mathcal{O}Y) \rightarrow KX_*(\mathcal{O}Y)
\]
is an isomorphism.
\end{proposition}
Higson and Roe \cite[Proposition 4.3]{MR1388312} proved this proposition
assuming that
the dimension of $Y$ is finite. Here we prove it without assuming that.
\begin{proof}
Any compact metric space can be embedded in the separable
 Hilbert space $l_2$. In fact, the stereographic projection gives an
 embedding in the unit ball of $l_2$. 
So we assume $Y \subset \{\vect{x} \in l_2 : \norm{\vect{x}} = 1\}$.
Then the open cone of $Y$ is given by
$\mathcal{O}Y = \{t\vect{x} \in l_2 : \vect{x} \in Y, t\in
 [0,\infty)\}$.
For $I\subset (0,\infty)$, set
\[
 Y \times I= \{t\vect{x} \in l_2 : \vect{x} \in Y, t\in I\}.
\]
Since $Y$ is compact, for each $n\in N$, there exist
 $p^n_1,\dots,p^n_{a_n} \in Y\times \{n\}$ such that
\begin{align}
\label{eq:1}
 \bigcup_{m=1}^{a_n}B(p^n_m,1) \supset Y\times \{n\}.
\end{align}
Here $B(x,r)$ denotes a ball of radius $r$ centered at $x$. 
Then we have
\[
 \bigcup_{m=1}^{a_n}B(p^n_m,2) \supset Y\times [n-1,n+1].
\]

For each $i\in \N$, we form a cover $\mathcal{U}_i$ of $\mathcal{O}Y$ as
 follows:
\begin{align*}
 U_m^n(i) &=  B(p_m^n, 3^i)\cap \mathcal{O}Y, \, m = 1,\dots, a_n, \\
 \mathcal{U}_i &= \bigcup_{n\geq 1}\{U_1^n(i), \dots, U_{a_n}^n(i)\}.
\end{align*}
It is clear that $\mathcal{U}_i$ is a locally finite cover and thus we
 obtain an anti-\v{C}ech system $\{\mathcal{U}_i\}_{i\geq 1}$. By the
 definition, it follows that
\[
 \bigcup_{n\geq 1}\bigcup_{m=1}^{a_n} B(p_m^n,3^i)\subset 
 \mathrm{Pen}(\mathcal{O}Y, 3^i).
\]
Then the method used in the proof of \cite[Proposition 4.3]{MR1388312}
 can be applied to $\{\mathcal{U}_i\}_{i\geq 1}$. This completes the
 proof of Proposition~\ref{prop:coarsenin_map}.
\end{proof}
\bibliographystyle{amsplain}
\bibliography{/Users/tomo_xi/Library/tex/books,/Users/tomo_xi/Library/tex/math,/Users/tomo_xi/Library/tex/preprints}

\bigskip
\address{ Tomohiro Fukaya \endgraf
Department of Mathematics, Kyoto University, Kyoto 606-8502, Japan}

\textit{E-mail address}: \texttt{tomo@math.kyoto-u.ac.jp}

\address{ Shin-ichi Oguni\endgraf
Department of Mathematics, Faculty of Science,
Ehime University,
2-5 Bunkyo-cho,
Matsuyama,
Ehime,
790-8577 Japan
}

\textit{E-mail address}: \texttt{oguni@math.sci.ehime-u.ac.jp}

\end{document}